\documentclass[a4paper,12pt]{amsart}
\usepackage{amssymb}
\usepackage{ifthen}
\usepackage[usenames]{color}
\usepackage{graphicx}
\usepackage[active]{srcltx}
 \makeatletter
\renewcommand*\subjclass[2][2000]{%
  \def\@subjclass{#2}%
  \@ifundefined{subjclassname@#1}{%
    \ClassWarning{\@classname}{Unknown edition (#1) of Mathematics
      Subject Classification; using '1991'.}%
  }{%
    \@xp\let\@xp\subjclassname\csname subjclassname@#1\endcsname
  }%
}
 \makeatother
\usepackage{enumerate,url,amssymb,  mathrsfs}%, pdfsync}% ,showkeys, pdfsync}

\newtheorem{theorem}{Theorem}[section]

\newtheorem{corollary}{Corollary}[section]
\newtheorem{claim}{Claim}[section]
\newtheorem{case}{Case}[section]
\newtheorem{lemma}{Lemma}[section]

\theoremstyle{definition}

\theoremstyle{remark}
\newtheorem{remark}{Remark}[section]

\numberwithin{equation}{section}

\def\XXint#1#2#3{{\setbox0=\hbox{$#1{#2#3}{\int}$}
\vcenter{\hbox{$#2#3$}}\kern-.5\wd0}}

{}
\def\le{\leqslant}
\def\ge{\geqslant}
\begin{document}

\title{Schwarz lemma for hyperbolic harmonic mappings in the unit ball}

\author{Jiaolong Chen}
\address{Jiaolong Chen, Key Laboratory of High Performance Computing and Stochastic Information Processing (HPCSIP)
(Ministry of Education of China), School of Mathematics and Statistics, Hunan Normal University, Changsha, Hunan 410081, People's Repulic of China}
\email{jiaolongchen@sina.com}

\author{David Kalaj}
\address{University of Montenegro, Faculty of Natural Sciences and
Mathematics, Cetinjski put b.b. 81000 Podgorica, Montenegro}
\email{davidkalaj@gmail.com}



\maketitle

%\def\thefootnote{}
%\footnotetext{
%\texttt{\tiny File:~\jobname .tex,
%          printed: \number\year-\number\month-\number\day,
%          \thehours.\ifnum\theminutes<10{0}\fi\theminutes}
%}
\makeatletter\def\thefootnote{\@arabic\c@footnote}\makeatother

\begin{abstract}
Assume that $p\in[1,\infty]$ and $u=P_{h}[\phi]$,
 where $\phi\in L^{p}(\mathbb{S}^{n-1},\mathbb{R}^n)$ and $u(0) = 0$.
Then we obtain the sharp inequality $|u(x)|\le G_p(|x|)\|\phi\|_{L^{p}}$ for some smooth  function $G_p$ vanishing at $0$.
Moreover, we obtain an explicit form of the sharp constant $C_p$ in the inequality $\|Du(0)\|\le C_p\|\phi\|_{L^{p}}$.
These two results generalize and extend some known result from harmonic mapping theory (\cite[Theorem 2.1]{kalaj2018}) and hyperbolic harmonic theory (\cite[Theorem 1]{bur}).
\end{abstract}

\maketitle
%\tableofcontents

\section{Introduction}\label{intsec}
 For $n\geq1$, let $\mathbb{R}^n$ be the standard Euclidean space with the norm $|x|=\sqrt{\sum_{i=1}^{n} x_i^2}$.
We use $\mathbb{B}^n$ and $\mathbb{B}_n$ to denote the unit ball in $\mathbb{R}^n$
and the unit ball in $\mathbb{C}^n\cong\mathbb{R}^{2n}$, respectively.
For $A=\big(a_{ij}\big)_{n\times n}\in \mathbb{R}^{n\times n}$,
 the matrix norm of $A$ is defined by
 $\|A\|=\sup\{|A\xi|:\; \xi\in \mathbb{S}^{n-1}\}$.

A mapping $u=(u_1,\cdots,u_n)\in C^{2}(\mathbb{B}^{n}, \mathbb{R}^{n})$ is said to be {\it hyperbolic harmonic} if
 $$\Delta_{h}u=(\Delta_{h}u_{1}, \cdots,\Delta_{h}u_{n})=0,$$  that is,
 for each $j\in \{1,\cdots, n\}$, $u_j$ satisfies the hyperbolic Laplace equation
$$
\Delta_{h}u_{j} (x)=(1-|x|^2)^2\Delta u_{j}(x)+2(n-2)(1-|x|^2)\sum_{i=1}^{n} x_{i} \frac{\partial u_{j}}{\partial x_{i}}(x)=0,
$$
 where
 $\Delta$ denotes the usual Laplacian in $\mathbb{R}^{n}$.
For convenience, in the rest of this paper, we call $\Delta_{h}$ the {\it hyperbolic Laplacian operator}.

When $n=2$, we easily see that hyperbolic harmonic mappings coincide with harmonic mappings.
In this paper, we focus our investigations on the case when $n\geq 3$.

For $p\in(0,\infty]$, the {\it Hardy space} $\mathcal{H}^{p}(\mathbb{B}^{n}, \mathbb{ R}^{n} )$ consists of all those mappings
$f: \mathbb{B}^{n}\rightarrow\mathbb{ R}^{n}$ such that $f$ is measurable, $M_{p}(r,f)$ exists for all $r\in (0,1)$ and $\|f\|_{\mathcal{H}^p}<\infty$, where
$$\|f\|_{\mathcal{H}^p}=\sup_{0<r<1} \big\{M_{p}(r,f)\big\}$$
and
$$\;\;M_{p}(r,f)=\begin{cases}
\displaystyle \;\left( \int_{\mathbb{S}^{n-1}} |f(r\xi)|^{p}d\sigma(\xi)\right)^{\frac{1}{p}} , & \text{ if } p\in (0,\infty),\\
\displaystyle \;\sup_{\xi\in \mathbb{S}^{n-1}} \big\{|f(r\xi)|\big\} , \;\;\;\;& \text{ if } p=\infty.
\end{cases}$$

\noindent Here and hereafter, $d \sigma$ always denotes the normalized surface measure on the unit sphere $\mathbb{S}^{n-1}$ in $\mathbb{R}^{n}$ so that $\sigma(\mathbb{S}^{n-1})=1$.

Similarly, for $p\in(0,\infty]$, we use $H^p(\mathbb{B}_{n},\mathbb{C}^{n})$ to denote
the Hardy space of mappings from $\mathbb{B}_n$ into $\mathbb{C}^{n}$.

If $\phi\in L^{1}(\mathbb{S}^{n-1},\mathbb{R}^{n})$, we define the {\it invariant Poisson integral} or {\it Poisson-Szeg\"{o} integral} of $\phi$ in $\mathbb{B}^{n}$ by (cf. \cite[Definition 5.3.2]{sto2016})
$$
P_{h}[\phi](x)=\int_{\mathbb{S}} P_{h}(x,\zeta)\phi(\zeta)d\sigma(\zeta),
$$
where
$$
 P_{h}(x,\zeta)=\left(\frac{1-|x|^2}{|x-\zeta|^{2}}\right)^{n-1}
$$
is the {\it Poisson-Szeg\"{o} kernel} with respective to $\Delta_{h}$ satisfying $$\int_{\mathbb{S}} P_{h}(x,\zeta) d\sigma(\zeta)=1$$
(cf. \cite[Lemma 5.3.1(c)]{sto2016}).
Similarly, if $\mu$ is a finite signed Borel measure in $\mathbb{S}^{n-1}$,
then invariant Poisson integral of $\mu$ will be denoted by $P_{h}[\mu]$, that is,
$$
P_{h}[\mu](x)=\int_{\mathbb{S}} P_{h}(x,\zeta) d\mu(\zeta).
$$
Furthermore, both $P_{h}[\phi]$ and $P_{h}[\mu]$ are hyperbolic harmonic in
$\mathbb{B}^{n}$ (cf. \cite{bur, sto2016}).

It is known that
if $\Delta_{h}u=0$ and $u\in \mathcal{H}^p(\mathbb{B}^n, \mathbb{R}^{n})$ with $1<p\leq\infty$,
 then $u$ has the following integral representation
 (cf. \cite[Theorem 7.1.1(c)]{sto2016})
$$
u(x)= P_{h}[\phi](x),
$$
where
$\phi\in L^{p}(\mathbb{S}^{n-1},\mathbb{R}^{n})$
is the boundary value of $u$ and
\begin{equation}\label{eq-1.1}
\|\phi\|_{L^p}= \|u\|_{\mathcal{H}^p}.
\end{equation}
If $\Delta_{h}u=0$ and $u\in \mathcal{H}^1(\mathbb{B}^n, \mathbb{R})$, then $u$ has the representation $u =P_{h}[\mu]$,
where $\mu$ is a signed Borel measure in $\mathbb{B}^n$.
Further, the similar arguments as \cite[Page 118]{ABR} show that
$\|u\|_{\mathcal{H}^1}=\|\mu\|$, where $\|\mu\|$ is the total variation of $\mu$ on $\mathbb{S}^{n-1}$.

In \cite{maro},  Macintyre and Rogosinski proved the following result:
Let $p\in[1,\infty]$ and $f$ be a holomorphic mapping in the unit disk $\mathbb{B}_1$ such that $f(0)=0$ and $\|f\|_{H^p}<\infty$,
then for $z\in \mathbb{B}_1$,
$$
|f(z)|\le \frac{|z|}{(1-|z|^2)^{1/p}}\|f\|_{H^p}
$$
with extremal functions $f(w)=\frac{A w}{(1-\bar z w)^{2/p}}$.
This is a generalization of Schwarz lemma (for $p=\infty$ it coincides with the classical Schwarz lemma).
For the high dimensional case, see \cite[Theorem~4.17]{zhu}.

The classical Schwarz lemma for harmonic mappings (\cite[Lemma 6.24]{ABR}) states that
if $f: \mathbb{B}^{n}\to \mathbb{R}^{n}$ is a bounded harmonic mapping with $f(0)=0$, then
$$
|f(x)|\le U(|x|e_{n})\|f\|_{\mathcal{H}^{\infty}}.
$$
Here $e_{n}=(0,\dots,0,1)\in \mathbb{S}^{n-1}$ and $U$ is a harmonic function of $\mathbb{B}^{n}$ into $[-1,1]$ defined by
$$
U(x)= P[\chi_{\mathbb{S}^{n-1}_{+}}-\chi_{\mathbb{S}^{n-1}_{-}}](x),
$$
where $\chi$ is the indicator function,
$\mathbb{S}^{n-1}_{+}=\{x\in \mathbb{S}^{n-1}: x_n \geq 0\},$ $\mathbb{S}^{n-1}_{-}=\{x\in \mathbb{S}^{n-1}: x_n \leq 0\}$ and
$P[\chi_{\mathbb{S}^{n-1}_{+}}-\chi_{\mathbb{S}^{n-1}_{-}}]$ is the Poisson integral of $\chi_{\mathbb{S}^{n-1}_{+}}-\chi_{\mathbb{S}^{n-1}_{-}}$ with respective to $\Delta$.
If $f:\mathbb{B}^{n}\rightarrow\mathbb{R}$ is a harmonic or hyperbolic harmonic mapping with $|f(x)|\leq1$ and $f(0)\in(-1,1)$, Burgeth \cite[Theorem 1]{bur} proved the following inequality
  $$
  m^{n}_{c}(|x|)\leq f(x)\leq M^{n}_{c}(|x|),
  $$
where $c=\frac{f(0)+1}{2}$, $m^{n}_{c}(|x|)$ and $M^{n}_{c}(|x|)$ are two functions in $\mathbb{B}^{n}$.
Recently, in \cite[Theorem 2.1]{kalaj2018}, the author find a sharp function $g_{p}$ for harmonic mappings $f$ in $\mathcal{H}^p(\mathbb{B}^{n},\mathbb{R}^{n})$ with $f(0)=0$:
For $p\in[1,\infty]$ and $x\in \mathbb{B}^{n}$,
$$
|f(x)|\le g_p(|x|)\|f\|_{\mathcal{H}^p}\;\;\text{and}\;\;
\|Df(0)\|\le n \left(\frac{\Gamma\left[\frac{n}{2}\right] \Gamma\left[\frac{1+q}{2}\right]}{\sqrt \pi\Gamma\left[\frac{n+q}{2}\right]}\right)^{\frac{1}{q}}\|f\|_{\mathcal{H}^p},
$$
where $q$ is the conjugate of $p$ and $Df(0):\mathbb{R}^n\to\mathbb{R}^n$ is the formal derivative.
See also \cite{kr} for related discussions.

 In this paper, we will establish the following counterpart of \cite[Theorem 2.1]{kalaj2018} in the setting of hyperbolic harmonic mappings in $\mathcal{H}^p(\mathbb{B}^{n},\mathbb{R}^{n})$.

\begin{theorem}\label{thm-1.1}
 Let $p\in[1,\infty]$, $q$ be its conjugate and for $r\in[0,1)$, define
\begin{equation}\label{eq-1.2}
G_p(r)=\left\{
            \begin{array}{ll}
              \inf_{a\in [0,\infty)}\sup_{\eta\in \mathbb{S}^{n-1}} |P_h(re_{n},\eta)-a|, & \hbox{if $q=\infty$;} \\
              \inf_{a\in [0,\infty)}\left(\int_{\mathbb{S}^{n-1}} |P_h(re_{n},\eta)-a|^qd\sigma(\eta)\right)^{1/q}, & \hbox{if $q\in[1,\infty)$}.
            \end{array}
          \right.
\end{equation}
 Suppose that $u=P_{h}[\phi]$ and $u(0)=0$, where
$\phi\in L^p(\mathbb{S}^{n-1},\mathbb{R}^{n})$.
 Then for any $x\in \mathbb{B}^n$,
\begin{equation}\label{eq-1.3}
|u(x)|\le G_p(|x|)\|\phi\|_{L^{p}}
\end{equation}
 and
 \begin{equation}\label{eq-1.4}
 \|Du(0)\|\le 2(n-1) \left(\frac{\Gamma\left[\frac{n}{2}\right] \Gamma\left[\frac{1+q}{2}\right]}{\sqrt \pi\Gamma\left[\frac{n+q}{2}\right]}\right)^{\frac{1}{q}}\|\phi\|_{L^{p}}.
 \end{equation}
Both inequalities \eqref{eq-1.3} and \eqref{eq-1.4} are sharp

In particular, if $p\in[1,\infty )$,
then $G_p$  is a increasing diffeomorphism of $[0,1)$ onto $[0,\infty)$ with $G_p(0)=0$;
if $p=\infty$,
then $G_{\infty}(r)=U_{h}(re_{n})$
and $G_{\infty}$  is an increasing diffeomorphism of $[0,1)$ onto itself,
where
$U_{h}=P_{h}[\chi_{\mathbb{S}_{+}^{n-1}} -\chi_{\mathbb{S}_{-}^{n-1}}]$.
\end{theorem}
\begin{remark}
\begin{enumerate}
\item It seems unlikely that we can explicitly express the function $G_p(r)$ for general $p$. However we demonstrate some special cases $p=1,2,\infty$ in Section~\ref{spc}.
\item Theorem \ref{thm-1.1} is generalization of
\cite[Theorem 1 and Corollary 2]{bur}.
\end{enumerate}

\end{remark}

\section{Proof of the main result}
The aim of this section is to prove the part of Theorem \ref{thm-1.1} when $p\in(1,\infty).$
In fact, it can be derived directly from Lemmas \ref{lem-2.5}$\sim$\ref{lem-2.8}.
Before the proofs of these lemmas, we need some preparation
which also consists of four lemmas.
The first one reads as follows.

\begin{lemma}\label{lem-2.1}
For $q\in[1,\infty)$, $r\in[0,1)$ and $a\in\mathbb{R}$,
\[\begin{split}
&\frac{\partial}{\partial a}\int_{\mathbb{S}^{n-1}} |P_h(re_{n},\eta)-a|^{q}\;d\sigma(\eta)
=\int_{\mathbb{S}^{n-1}}
q\big(a-P_h(re_{n},\eta)\big) |P_h(re_{n},\eta)-a|^{q-2}\;d\sigma(\eta).
\end{split}\]
\end{lemma}

\begin{proof}
We consider the case when $q\in[1,2)$ and the case when $q\in [2,\infty)$, separately.

\begin{case}\label{case-2.1} $q\in[1,2)$. \end{case}
For $(r,a)\in[0,1)\times\mathbb{R}$,
since
$$
\int_{\mathbb{S}^{n-1}} \frac{\partial}{\partial a}
|P_h(re_{n},\eta)-a|^{q}\;d\sigma(\eta)
=\int_{\mathbb{S}^{n-1}}
q\big(a-P_h(re_{n},\eta)\big) |P_h(re_{n},\eta)-a|^{q-2}\;d\sigma(\eta)
$$
and
$$
 \int_{\mathbb{S}^{n-1}} |P_h(re_{n},\eta)-a|^{q-1}\;d\sigma(\eta)\leq
 \left( \frac{1+r}{1-r}\right)^{(n-1)(q-1)}+|a|^{q-1},
 $$
 then by \cite[Proposition 2.4]{kp} or \cite{ta}, we obtain that
 \[\begin{split}
 \frac{\partial}{\partial a}\int_{\mathbb{S}^{n-1}} |P_h(re_{n},\eta)-a|^{q}\;d\sigma(\eta)
=\int_{\mathbb{S}^{n-1}}
q\big(a-P_h(re_{n},\eta)\big) |P_h(re_{n},\eta)-a|^{q-2}\;d\sigma(\eta).
\end{split}\]
 \begin{case}\label{case-2.2} $p\in [2,\infty)$. \end{case}

By direct calculations, we have
$$
\frac{\partial}{\partial a}
 |P_h(re_{n},\eta)-a|^{q}
=q\big(a-P_h(re_{n},\eta)\big)|P_h(re_{n},\eta)-a|^{q-2}.
$$
Obviously, the mappings
$$
(r,a,\eta)\mapsto  |P_h(re_{n},\eta)-a|^{q}
\quad\text{and}\quad
(r,a,\eta)\mapsto\frac{\partial}{\partial a}|P_h(re_{n},\eta)-a|^{q}
$$
are continuous in $[0,1)\times \mathbb{R}\times\mathbb{S}^{n-1}$.
Therefore, for any $(r,a)\in[0,1)\times \mathbb{R}$,
\[\begin{split}
&\frac{\partial}{\partial a}\int_{\mathbb{S}^{n-1}} |P_h(re_{n},\eta)-a|^{q}\;d\sigma(\eta)
=\int_{\mathbb{S}^{n-1}}
\frac{\partial}{\partial a}|P_h(re_{n},\eta)-a|^{q}\;d\sigma(\eta).
\end{split}\]
as required.
The proof of the lemma is completed.
\end{proof}

For $q\in(1,\infty)$, $r\in(0,1)$ and $a\in \mathbb{R}$, let
\begin{equation}\label{eq2.01}
F(r,a)=\int_{\mathbb{S}^{n-1}} \big(P_h(re_{n},\eta)-a\big)|P_h(re_{n},\eta)-a|^{q-2} d\sigma(\eta).
\end{equation}
Then we have the following results on $F(r,a)$.
\begin{lemma}\label{lem-2.2}
For $q\in(1,\infty)$, $r\in(0,1)$ and $a\in \mathbb{R}$,
\begin{equation}\label{eq-2.2}
\partial_aF(r,a)=(1-q)\int_{\mathbb{S}^{n-1}} |P_h(re_{n},\eta)-a|^{q-2} d\sigma(\eta)
\end{equation}
and
\begin{equation}\label{eq-2.3}
\partial_rF(r,a)=(q-1)\int_{\mathbb{S}^{n-1}} \partial_{ r} P_h(re_{n},\eta)\cdot |P_h(re_{n},\eta)-a|^{q-2} d\sigma(\eta).
\end{equation}
Furthermore, for any $[\mu_{1},\mu_{2}]\subset (0,1)$ and $[\nu_{1},\nu_{2}]\subset (0,\infty)$,
both $\partial_aF(r,a)$ and $\partial_rF(r,a)$ are
uniformly convergent w.r.t. $(r,a)\in[\mu_{1},\mu_{2}]\times [\nu_{1},\nu_{2}]$.

\end{lemma}
\begin{proof}
In order to prove this lemma, we only need to prove \eqref{eq-2.2} and the uniformly convergence of $\partial_aF(r,a)$ since \eqref{eq-2.3}  and the uniformly convergence of $\partial_rF(r,a)$ can be proved in a similar way.
For this, we consider the case when $q\in(1,2)$ and the case when $q\in [2,\infty)$, separately.

\begin{case}\label{case-2.3} $q\in(1,2)$. \end{case}
For fixed $r\in(0,1)$ and $\eta=(\eta_{1},\ldots,\eta_{n})\in\mathbb{S}^{n-1}$,
by calculations, we know that
\begin{eqnarray}\label{eq-2.4}
&\;\;&-\int_{\mathbb{S}^{n-1}} \frac{\partial}{\partial a}(P_h(re_{n},\eta)-a)|P_h(re_{n},\eta)-a|^{q-2}\;d\sigma(\eta)\\\nonumber
&=&(q-1)\int_{\mathbb{S}^{n-1}} |P_h(re_{n},\eta)-a|^{q-2} d\sigma(\eta)
\leq4^{n-1}I(r,a),
\end{eqnarray}
where
$$
I(r,a)=\int_{\mathbb{S}^{n-1}} \big|(1-r^{2})^{n-1}-a(1+r^{2}-2r\eta_{n})^{n-1} \big|^{q-2}\;d\sigma(\eta).
 $$

 If $a\leq0$, then
\begin{eqnarray}\label{eq-2.5}
I(r,a)&=&\int_{\mathbb{S}^{n-1}} \big((1-r^{2})^{n-1}+|a|(1+r^{2}-2r\eta_{n})^{n-1} \big)^{q-2}\;d\sigma(\eta)\\\nonumber
&\leq&\big( (1-r^{2})^{n-1}+|a|(1-r)^{2n-2} \big)^{q-2}.
\end{eqnarray}

 If $a>0$, by using the spherical coordinates
 (cf. \cite[Section 2.2]{chen2018}), we obtain
\begin{eqnarray}\label{eq-2.6}
I(r,a)&=&\int_{0}^{\pi} \sin^{n-2}\theta \left|(1-r^{2})^{n-1}-a ( 1+r^{2} -2r \cos\theta )^{n-1} \right|^{q-2}\;d\theta\\\nonumber
&=&I_{1}(r,a)+I_{2}(r,a),
\end{eqnarray}
where
 $$
 I_{1}(r,a)=\int_{0}^{\frac{\pi}{2}} \sin^{n-2}\theta \left|(1-r^{2})^{n-1}- a (1+r^{2} -2r \cos\theta )^{n-1} \right|^{q-2}\;d\theta
 $$
 and
  $$
 I_{2}(r,a)=\int_{\frac{\pi}{2}}^{\pi} \sin^{n-2}\theta \left|(1-r^{2})^{n-1}- a (1+r^{2} -2r \cos\theta )^{n-1} \right|^{q-2}\;d\theta.
 $$

In the following, we estimate $I_{1}(r,a)$ and $I_{2}(r,a)$, respectively.

\begin{claim}\label{claim-2.1} $I_{1}(r,a)$ is convergent in $(0,1)\times(0,\infty)$ and uniformly convergent in $[\mu_{1},\mu_{2}]\times [\nu_{1},\nu_{2}]$ for any $[\mu_{1},\mu_{2}]\subset (0,1)$ and $[\nu_{1},\nu_{2}]\subset (0,\infty)$.
\end{claim}

For $(r,a)\in(0,1)\times(0,\infty)$, obviously,
  \begin{eqnarray*}
 \;\;I_{1}(r,a)
 &=&\int_{0}^{1} (1-x^{2})^{\frac{n-3}{2}} \left|(1-r^{2})^{n-1}-\left(a^{\frac{1}{n-1}}(1+r^{2}-2rx) \right)^{n-1} \right|^{q-2}\;dx\\
&\leq&\int_{0}^{1}  \left|(1-r^{2})^{n-1}-\left(a^{\frac{1}{n-1}}(1+r^{2}-2rx) \right)^{n-1} \right|^{q-2}\;dx.
\end{eqnarray*}
 Moreover, for $x\in[0,1]$, $r\in(0,1)$ and $a\in(0,\infty)$,
\begin{eqnarray*}
&\;\;& \left|(1-r^{2})^{n-1}-\left(a^{\frac{1}{n-1}}(1+r^{2}-2rx) \right)^{n-1} \right|\\
&=&\left| 1-r^{2} -a^{\frac{1}{n-1}}(1+r^{2}-2rx) \right|\cdot\sum_{i=0}^{n-2}(1-r^{2})^{i}\left(a^{\frac{1}{n-1}}(1+r^{2}-2rx) \right)^{n-2-i}\\
&\geq&A(a,r)\cdot \left| 1-r^{2} -a^{\frac{1}{n-1}}(1+r^{2}-2rx) \right|,
\end{eqnarray*}
where $A(a,r)=(n-1)\min\{a^{\frac{n-2}{n-1}},1 \}(1-r)^{2n-4}$.
Then for any $(r,a)\in(0,1)\times(0,\infty)$, elementary calculations lead to
\begin{eqnarray}\label{eq-2.7}
&\;\;& I_{1}(r,a) \\\nonumber
&\leq&\frac{A^{q-2}(a,r)
}{2r(q-1)a^{\frac{1}{n-1}}} \left( \left| 1-r^{2} -a^{\frac{1}{n-1}}(1-r)^{2} \right|^{q-1}+\left| 1-r^{2} -a^{\frac{1}{n-1}}(1+r^{2}) \right|^{q-1}\right).
 \end{eqnarray}

 Let
$$
 1-r^{2} -a^{\frac{1}{n-1}}(1+r^{2}-2r\lambda_{1})=0.
$$
Then
\begin{eqnarray}\label{eq-2.8}
\lambda_{1}=\frac{1+r^{2}}{2r}-\frac{1-r^{2}}{2r}a^{\frac{1}{1-n}}.
\end{eqnarray}

For any $\delta>0$, $r\in(0,1)$ and $a>0$, since
\begin{eqnarray*}
&\;\;& A^{q-2}(a,r)\cdot\int_{\lambda_{1}-\delta}^{\lambda_{1}}  \left| 1-r^{2} -a^{\frac{1}{n-1}}(1+r^{2}-2rx) \right|^{q-2}dx\\
&\leq& A^{q-2}(a,r)\big(2ra^{\frac{1}{n-1}}  \big)^{q-2}
\int_{\lambda_{1}-\delta}^{\lambda_{1}}  |x-\lambda_{1} |^{q-2}dx
=A^{q-2}(a,r)\big(2ra^{\frac{1}{n-1}}  \big)^{q-2} \delta^{q-1}/(q-1)
\end{eqnarray*}
and
\begin{eqnarray*}
&\;\;& A^{q-2}(a,r)\cdot\int_{\lambda_{1}}^{\lambda_{1}+\delta}  \left| 1-r^{2} -a^{\frac{1}{n-1}}(1+r^{2}-2rx) \right|^{q-2}dx\\
&\leq& A^{q-2}(a,r)\big(2ra^{\frac{1}{n-1}}  \big)^{q-2}
\int_{\lambda_{1}}^{\lambda_{1}+\delta}  |x-\lambda_{1} |^{q-2}dx
=A^{q-2}(a,r)\big(2ra^{\frac{1}{n-1}}  \big)^{q-2} \delta^{q-1}/(q-1),
\end{eqnarray*}
we see that $I_{1}(r,a)$ is uniformly convergent w.r.t. $(r,a)\in[\mu_{1},\mu_{2}]\times [\nu_{1},\nu_{2}]$ for any $[\mu_{1},\mu_{2}]\subset (0,1)$ and $[\nu_{1},\nu_{2}]\subset (0,\infty)$.
Hence, Claim  \ref{claim-2.1} is proved.

\begin{claim}\label{claim-2.2} $I_{2}(r,a)$ is convergent in $(0,1)\times(0,\infty)$ and uniformly convergent in $[\mu_{1},\mu_{2}]\times [\nu_{1},\nu_{2}]$ for any $[\mu_{1},\mu_{2}]\subset (0,1)$ and $[\nu_{1},\nu_{2}]\subset (0,\infty)$.
\end{claim}

Similar arguments as in the proof of Claim \ref{claim-2.1} guarantee that
$$
 I_{2}(r,a)
\leq\int_{0}^{1}  \left|(1-r^{2})^{n-1}-\left(a^{\frac{1}{n-1}}(1+r^{2}+2rx) \right)^{n-1} \right|^{q-2}\;dx.
$$
 Moreover, for $x\in[0,1]$, $r\in(0,1)$ and $a\in(0,\infty)$,
$$
 \left|(1-r^{2})^{n-1}-\left(a^{\frac{1}{n-1}}(1+r^{2}+2rx) \right)^{n-1} \right|
 \geq B(a,r)\cdot \left| 1-r^{2} -a^{\frac{1}{n-1}}(1+r^{2}+2rx) \right|,
$$
where $B(a,r)=(n-1)\min\{a^{\frac{n-2}{n-1}},1 \}(1-r^{2})^{n-2}$.
Therefore, by elementary calculations, we obtain
\begin{eqnarray}\label{eq-2.9}
&\;\;&I_{2}(r,a)\\\nonumber
&\leq&\frac{B^{q-2}(a,r)
}{2r(q-1)a^{\frac{1}{n-1}}} \left( \left|1-r^{2}-a^{\frac{1}{n-1}}(1+r)^{2} \right|^{q-1}+\left|1-r^{2}-a^{\frac{1}{n-1}}(1+r^{2}) \right|^{q-1}\right).
\end{eqnarray}
for any $r\in(0,1)$ and $a\in(0,\infty)$.

Let
$$
 1-r^{2} -a^{\frac{1}{n-1}}(1+r^{2}+2r\lambda_{2})=0.
$$
Then
$$
\lambda_{2}=\frac{1-r^{2}}{2r}a^{\frac{1}{1-n}}-\frac{1+r^{2}}{2r}.
$$
For any $\delta>0$, $r\in(0,1)$ and $a>0$, since
\begin{eqnarray*}
&\;\;& B^{q-2}(a,r)\cdot\int_{\lambda_{2}-\delta}^{\lambda_{2}}  \left| 1-r^{2} -a^{\frac{1}{n-1}}(1+r^{2}+2rx) \right|^{q-2}dx\\
&\leq& B^{q-2}(a,r)\big(2ra^{\frac{1}{n-1}}  \big)^{q-2}
\int_{\lambda_{2}-\delta}^{\lambda_{2}}  |x-\lambda_{2} |^{q-2}dx
=B^{q-2}(a,r)\big(2ra^{\frac{1}{n-1}}  \big)^{q-2} \delta^{q-1}/(q-1)
\end{eqnarray*}
and
\begin{eqnarray*}
&\;\;& B^{q-2}(a,r)\cdot\int_{\lambda_{2}}^{\lambda_{2}+\delta}  \left| 1-r^{2} -a^{\frac{1}{n-1}}(1+r^{2}+2rx) \right|^{q-2}dx\\
&\leq& B^{q-2}(a,r)\big(2ra^{\frac{1}{n-1}}  \big)^{q-2}
\int_{\lambda_{2}}^{\lambda_{2}+\delta}  |x-\lambda_{2} |^{q-2}dx
=B^{q-2}(a,r)\big(2ra^{\frac{1}{n-1}}  \big)^{q-2} \delta^{q-1}/(q-1),
\end{eqnarray*}
we see that $I_{2}(r,a)$ is uniformly convergent w.r.t. $(r,a)\in[\mu_{1},\mu_{2}]\times [\nu_{1},\nu_{2}]$ for any $[\mu_{1},\mu_{2}]\subset (0,1)$ and $[\nu_{1},\nu_{2}]\subset (0,\infty)$.
 Hence, Claim  \ref{claim-2.2} is proved.

 \medskip

Now, we prove \eqref{eq-2.2}.
For any $r\in(0,1)$ and $a\in(0,\infty)$,
by \eqref{eq-2.6}, \eqref{eq-2.7} and \eqref{eq-2.9}, we get
\begin{eqnarray*}
I(r,a)
&\leq& I_{1}(r,a)+ I_{2}(r,a)\\
&\leq&\frac{2\big(A^{q-2}(a,r)+ B^{q-2}(a,r)\big)}
{(q-1)a^{\frac{1}{n-1}}r} \left( (1-r^{2})^{q-1}+a^{\frac{q-1}{n-1}}(1+r)^{2q-2}\right).
\end{eqnarray*}
From this, together with \eqref{eq-2.5} and \cite[Proposition 2.4]{kp} or \cite{ta},
we see that \eqref{eq-2.2} is true.
By \eqref{eq-2.2}, \eqref{eq-2.4}, Claims \ref{claim-2.1} and \ref{claim-2.2}, we obtain that $\partial_aF(r,a)$ is
uniformly convergent w.r.t. $(r,a)\in[\mu_{1},\mu_{2}]\times [\nu_{1},\nu_{2}]$ for any $[\mu_{1},\mu_{2}]\subset (0,1)$ and $[\nu_{1},\nu_{2}]\subset (0,\infty)$.

 \begin{case}\label{case-2.4} $p\in [2,\infty)$. \end{case}

By direct calculations, we have
$$
\frac{\partial}{\partial a}
(P_h(re_{n},\eta)-a)|P_h(re_{n},\eta)-a|^{q-2}
=(1-q) |P_h(re_{n},\eta)-a|^{q-2}.
$$
Obviously, the mappings
$$
(r,a,\eta)\mapsto  (P_h(re_{n},\eta)-a)|P_h(re_{n},\eta)-a|^{q-2}
$$
and
$$
(r,a,\eta)\mapsto(1-q) |P_h(re_{n},\eta)-a|^{q-2}
$$
are continuous in $(0,1)\times\mathbb{R}\times\mathbb{S}^{n-1}$.
Therefore, for any $(r,a)\in (0,1)\times\mathbb{R}$,
\eqref{eq-2.2} is true and  $\partial_aF(r,a)$ is
uniformly convergent w.r.t. $(r,a)\in[\mu_{1},\mu_{2}]\times [\nu_{1},\nu_{2}]$ for any $[\mu_{1},\mu_{2}]\subset (0,1)$ and $[\nu_{1},\nu_{2}]\subset (0,\infty)$.
The proof of the lemma is completed.
\end{proof}

\begin{lemma}\label{lem-2.3}
For $q\in(1,\infty)$,
 both $\partial_aF(r,a)$ and
$\partial_rF(r,a)$ are continuous
w.r.t. $(r,a)\in(0,1)\times(0,\infty)$.
\end{lemma}
\begin{proof}
In order to prove this lemma, we only need to prove the continuity of $\partial_aF(r,a)$ since the continuity of $\partial_rF(r,a)$ can be proved in a similar way.
For this, we consider the case when $q\in(1,2)$ and the case when $q\in [2,\infty)$, separately.

\begin{case}\label{case-2.5} $q\in(1,2)$. \end{case}
In order to check the continuity of $\partial_aF(r,a)$, we only need to prove that $\partial_aF(r,a)$ is continuous at every fixed point
$(r_{0},a_{0})\in(0,1)\times(0,\infty)$.
Assume that
$(r_{0},a_{0})\in(\mu_{1},\mu_{2})\times (\nu_{1},\nu_{2})
\subset (0,1)\times(0,\infty)$
and $(r_{0}+\Delta r,a_{0}+\Delta a)
\in(\mu_{1},\mu_{2})\times (\nu_{1},\nu_{2})$.

For $(r,a,x)\in[\mu_{1},\mu_{2}]\times [\nu_{1},\nu_{2}]\times [-1,1]$, let
\begin{eqnarray}\label{eq-2.10}
\lambda_{3}(r,a,x)= \frac{(q-1)(1-x^2)^{\frac{n-3}{2}}\cdot(1+r^{2}-2rx)^{(n-1)(2-q)}}
{\left(\sum_{i=0}^{n-2}a^{\frac{i}{n-1}}(1-r^{2})^{n-2-i}(1+r^{2}-
2rx)^{i}\right)^{2-q}}
\end{eqnarray}
and
\begin{eqnarray}\label{eq-2.11}
\lambda_{4}(r,a,x)
=\big(2r a^{\frac{1}{n-1}}\big)^{(2-q)(n-1)}\big(x-\lambda_{1}\big)^{2-q},
\end{eqnarray}
where $\lambda_{1}=\lambda_{1}(r,a)$ is the constant from \eqref{eq-2.8}
and $\lambda_{1}(r,a)$ means that the constant $\lambda_{1}$ depends only on $r$ and $a$.
Obviously, $\lambda_{3}(r,a,x)$ is continuous in $[\mu_{1},\mu_{2}]\times [\nu_{1},\nu_{2}]\times [-1,1]$.
It follows from spherical coordinate transformation and \eqref{eq-2.2} that
$$
-\partial_aF(r,a)
=(q-1)\int_{\mathbb{S}^{n-1}} |P_h(re_{n},\eta)-a|^{q-2} d\sigma(\eta)
=J_{1}(r,a)+J_{2}(r,a),
 $$
 where
$$
J_{1}(r,a)
= \int_{0}^{1}
\frac{(q-1) (1-x^2)^{\frac{n-3}{2}}\cdot(1+r^{2}-2rx)^{(n-1)(2-q)}}
{|(1-r^{2})^{n-1}-(a^{\frac{1}{n-1}}(1+r^{2})-
2r a^{\frac{1}{n-1}}x)^{n-1}|^{2-q}} \;dx
=\int_{0}^{1}
\frac{\lambda_{3}(r,a,x)}
{\lambda_{4}(r,a,x)}\;dx
$$
and
$$
J_{2}(r,a)
=\int_{0}^{1}
\frac{(q-1)(1-x^2)^{\frac{n-3}{2}}\cdot(1+r^{2}+2rx)^{(n-1)(2-q)}}
{|(1-r^{2})^{n-1}-(a^{\frac{1}{n-1}}(1+r^{2})+
2r a^{\frac{1}{n-1}}x)^{n-1}|^{2-q}} \;dx
= \int_{0}^{1}
\frac{\lambda_{3}(r,a,-x)}
{\lambda_{4}(r,a,-x)}\;dx.
$$

By Claim \ref{claim-2.1}, we know that
$J_{1}(r,a)$ is
uniformly convergent in $[\mu_{1},\mu_{2}]\times [\nu_{1},\nu_{2}]$ for any $[\mu_{1},\mu_{2}]\subset (0,1)$ and $[\nu_{1},\nu_{2}]\subset (0,\infty)$.
Without loss of generalization, we assume that $\lambda_{1}(r_{0},a_{0})\in(0,1)$.
Then for any $\varepsilon_{1}>0$, there exist constants
$\iota_{1}= \iota_{1}(\varepsilon_{1})\rightarrow 0^{+}$ and
$\iota_{2}= \iota_{2}(\varepsilon_{1})\rightarrow 0^{+}$ such that for any $(r,a)\in[\mu_{1},\mu_{2}]\times [\nu_{1},\nu_{2}]$,
 $$
 \left|
\int_{\lambda_{1}(r_{0},a_{0})-\iota_{1}}^{\lambda_{1}(r_{0},a_{0})+\iota_{2}}
\frac{\lambda_{3}(r,a,x)}
{\lambda_{4}(r,a,x)}\;dx\right|<\varepsilon_{1}.
$$
 Then
\begin{eqnarray}\label{eq-2.12}
&\;\;& \left|J_{1}(r_{0}+\Delta r,a_{0}+\Delta a)-J_{1}(r_{0},a_{0})\right|\\\nonumber
&=&\left|\int_{0}^{\lambda_{1}(r_{0},a_{0})-\iota_{1}}
\left(\frac{\lambda_{1}(r_{0}+\Delta r,a_{0}+\Delta a,x)}
{\lambda_{2}(r_{0}+\Delta r,a_{0}+\Delta a,x)}-
\frac{\lambda_{1}(r_{0},a_{0},x)}
{\lambda_{2}(r_{0},a_{0},x)}\right)\;dx\right.\\\nonumber
&\;\;& +\int_{\lambda_{1}(r_{0},a_{0})-\iota_{1}}^{\lambda_{1}(r_{0},a_{0})+\iota_{2}}
\left(\frac{\lambda_{1}(r_{0}+\Delta r,a_{0}+\Delta a,x)}
{\lambda_{2}(r_{0}+\Delta r,a_{0}+\Delta a,x)}-\frac{\lambda_{1}(r_{0},a_{0},x)}
{\lambda_{2}(r_{0},a_{0},x)}\right)\;dx \\\nonumber
&\;\;&\left.+\int_{\lambda_{1}(r_{0},a_{0})+\iota_{2}}^{1}
\left(\frac{\lambda_{1}(r_{0}+\Delta r,a_{0}+\Delta a,x)}
{\lambda_{2}(r_{0}+\Delta r,a_{0}+\Delta a,x)}-\frac{\lambda_{1}(r_{0},a_{0},x)}
{\lambda_{2}(r_{0},a_{0},x)}\right)\;dx\right|\\\nonumber
&\leq&\left|\int_{0}^{\lambda_{1}(r_{0},a_{0})-\iota_{1}}
\left(\frac{\lambda_{1}(r_{0}+\Delta r,a_{0}+\Delta a,x)}
{\lambda_{2}(r_{0}+\Delta r,a_{0}+\Delta a,x)}-
\frac{\lambda_{1}(r_{0},a_{0},x)}
{\lambda_{2}(r_{0},a_{0},x)}\right)\;dx\right.\\\nonumber
&\;\;&\left.+\int_{\lambda_{1}(r_{0},a_{0})+\iota_{2}}^{1}
\left(\frac{\lambda_{1}(r_{0}+\Delta r,a_{0}+\Delta a,x)}
{\lambda_{2}(r_{0}+\Delta r,a_{0}+\Delta a,x)}-\frac{\lambda_{1}(r_{0},a_{0},x)}
{\lambda_{2}(r_{0},a_{0},x)}\right)\;dx\right|+ \varepsilon_{1}.
\end{eqnarray}

By \eqref{eq-2.8}, it is easy to see that $\lambda_{1}(r,a)$ is uniformly continuous in $[\mu_{1},\mu_{2}]\times [\nu_{1},\nu_{2}]$.
Then for any $\iota'\in(0,\min\{\iota_{1},\iota_{2}\})$, there exist a constant
$\iota_{3}= \iota_{3}(\iota')\rightarrow 0^{+}$ such that for any $(r,a)\in[r_{0}-\iota_{3},r_{0}+\iota_{3}]\times[a_{0}-\iota_{3},a_{0}+\iota_{3}]
\subset [\mu_{1},\mu_{2}]\times [\nu_{1},\nu_{2}]$,
$$
\lambda_{1}(r,a)
\in\left(\lambda_{1}(r_{0},a_{0})-\frac{\iota'}{2},\lambda_{1}(r_{0},a_{0})+\frac{\iota'}{2}\right)
\subset\big(\lambda_{1}(r_{0},a_{0})-\iota_{1},\lambda_{1}(r_{0},a_{0})+\iota_{2}\big).
$$
This, together with \eqref{eq-2.10} and \eqref{eq-2.11}, implies that
the mapping $(r,a,x)\mapsto \frac{\lambda_{3}(r,a,x)}
{\lambda_{4}(r,a,x)}$ is continuous (also uniformly continuous) in
$$[r_{0}- \iota_{3},r_{0}+ \iota_{3}]
\times  [a_{0}- \iota_{3},a_{0}+ \iota_{3}]
\times  [0,\lambda_{1}(r_{0},a_{0})-\iota_{1}]$$
and
$$[r_{0}- \iota_{3},r_{0}+ \iota_{3}]
\times  [a_{0}- \iota_{3},a_{0}+ \iota_{3}]
\times  [\lambda_{1}(r_{0},a_{0})+\iota_{2},1],$$
respectively.
Therefore, there exists $\iota_{4}=\iota_{4}(\varepsilon_{1})\leq\iota_{3}$
such that for all $|\Delta r|<\iota_{4}$, $|\Delta a|<\iota_{4}$
 and for all $x\in [0,\lambda_{1}(r_{0},a_{0})-\iota_{1}]\cup [\lambda_{1}(r_{0},a_{0})+\iota_{2},1]$,
\begin{eqnarray}\label{eq-2.13}
\left| \frac{\lambda_{3}(r_{0}+\Delta r,a_{0}+\Delta a,x)}
{\lambda_{4}(r_{0}+\Delta r,a_{0}+\Delta a,x)}-\frac{\lambda_{3}(r_{0},a_{0},x)}
{\lambda_{4}(r_{0},a_{0},x)} \right|<\varepsilon_{1} .
\end{eqnarray}
Then by \eqref{eq-2.12} and \eqref{eq-2.13}, we obtain
 $$
  \left|J_{1}(r_{0}+\Delta r,a_{0}+\Delta a)-J_{1}(r_{0},a_{0})\right|
  \leq 2\varepsilon_{1},
 $$
which means that $J_{1}$ is continuous at $(r_{0},a_{0})$.

\begin{case}\label{case-2.6} $q\in[2,\infty)$. \end{case}
Obviously, the mapping
$$(r,a,\eta)\mapsto(1-q) |P_h(re_{n},\eta)-a|^{q-2}$$
is continuous in $(0,1)\times(0,\infty)\times\mathbb{S}^{n-1}$.
Then by \eqref{eq-2.2}, we know that $\partial_aF(r,a)$ is continuous
w.r.t. $(r,a)\in(0,1)\times(0,\infty)$.
\end{proof}

For $q\in[1,\infty)$, $r\in[0,1)$ and $a\in \mathbb{R}$, define
\begin{equation}\label{eq-2.14}
 \Phi_{q,r}(a)=\left(\int_{\mathbb{S}^{n-1}} |P_h(re_{n},\eta)-a|^qd\sigma(\eta)\right)^{1/q}.
\end{equation}
 Then by Lemmas \ref{lem-2.1}$\sim$\ref{lem-2.3}, we obtain the following result for $\Phi_{q,r}(a)$.
\begin{lemma}\label{lem-2.4}
For $q\in(1,\infty)$ and $r\in[0,1)$,
there is a unique constant $a^{*}=a(r)\in(0,\infty)$ such that $$\Phi_{q,r}(a^*)=\min_{a\in \mathbb{R}}\Phi_{q,r}(a),$$
where $a(0)=1$ and
$a(r)$ is a smooth function in $(0,1)$.
\end{lemma}

\begin{proof}
When $r=0$, it follows from \eqref{eq-2.14} that $a^{*}=1$.
Hence, to prove the lemma, it remains to consider the case when $r\in(0,1)$.

For $r\in(0,1)$, $t\in(0,1)$ and $b,c\in\mathbb{R}$ with $b\not=c$, by Minkowski inequality, we obtain
$$\Phi_{q,r}\big(\lambda b+(1-t)c\big)
<t \Phi_{q,r}(b)+(1-t) \Phi_{q,r}(c),$$
which means that $\Phi_{q,r}(a)$ is strictly convex in $\mathbb{R}$.
Furthermore, by Lemma \ref{lem-2.1}, we know that for $q\in(1,\infty)$, $r\in(0,1)$ and $a\in \mathbb{R}$,
\begin{equation}\label{eq-2.15}
\frac{d}{da}\Phi_{q,r}(a)
=-\left(\int_{\mathbb{S}^{n-1}} |P_h(re_{n},\eta)-a|^{q}d\sigma(\eta)\right)^{1/q-1}F(r,a),
\end{equation}
where $F(r,a)$ is the mapping from \eqref{eq2.01}.
Therefore,
$$
\frac{d}{da}\Phi_{q,r}(0)=-\int_{\mathbb{S}^{n-1}} |P_h(re_{n},\eta)|^{q-1}d\sigma(\eta)
\left(\int_{\mathbb{S}^{n-1}}  |P_h(re_{n},\eta)|^{q}d\sigma(\eta)\right)^{1/q-1}<0.
$$
These, together with the fact $\lim_{a\rightarrow\infty}\Phi_{q,r}(a)=\infty$,
show that for any $r\in(0,1),$
$\Phi_{q,r}(a)$ has only one stationary point in $(0,\infty)$ which is its minimum, i.e., $a^{*}=a(r)$.

By \eqref{eq-2.15}, we see that for any $r\in(0,1)$, $\frac{d}{da}\Phi_{q,r}\big(a(r)\big)=0$ is equivalent to
\begin{equation}\label{eq-2.16}
F\big(r,a(r)\big)=0.
\end{equation}
Furthermore, Lemmas \ref{lem-2.2} and \ref{lem-2.3} tell us that both $\partial_aF(r,a)$ and $\partial_rF(r,a)$ are continuous w.r.t. $(r,a)\in(0,1)\times (0,\infty)$ and that $\partial_aF(r,a)<0$.
Therefore, it follows from \eqref{eq-2.16} and the implicit function theorem
that $a^*=a(r)$ is a smooth function w.r.t. $r\in(0,1)$ and
$$
\frac{d a(r)}{d r}
=-\frac{\frac{\partial F(r,a)}{\partial r}}
{\frac{\partial F(r,a)}{\partial a}}.
$$
The proof of the lemma is completed.
\end{proof}

Based on Lemma \ref{lem-2.4}, we have the following estimate on $|u|$.
 \begin{lemma}\label{lem-2.5}
For $p\in(1,\infty)$, suppose $u=P_{h}[\phi]$ and $u(0)=0$, where
$\phi\in L^p(\mathbb{S}^{n-1},\mathbb{R}^{n})$.
Then
\begin{equation}\label{eq-2.17}
|u(x)|\le G_p(|x|)\|\phi\|_{L^{p}}
\end{equation}
 in $\mathbb{B}^n$,
 where $G_p $ is the mapping from Theorem \ref{thm-1.1} and  it  is smooth in $(0,1)$ with $G_p(0)=0$ .
The inequality  is sharp.
 \end{lemma}

\begin{proof} Let $p\in(1,\infty)$ and $q$ be its conjugate.
For any $x\in \mathbb{B}^{n}$ and $a\in \mathbb{R}^{n}$, it follows from the assumption $u=P_{h}[\phi]$ and $u(0)=0$ that
$$
u(x)=\int_{\mathbb{S}^{n-1}} (P_{h}(x,\eta)-a)\phi(\eta)d\sigma(\eta).
$$

If $x=re_{n}$ for some $r\in[0,1)$, then by using H\"older's inequality, we have
\begin{equation}\label{eq-2.18}
|u(x)|\leq \left(\int_{\mathbb{S}^{n-1}} |P_h(re_{n},\eta)-a|^qd\sigma(\eta)\right)^{1/q}\|\phi\|_{L^{p}}= \Phi_{q,r}(a)\cdot\|\phi\|_{L^{p}}
\end{equation}
for any $a\in \mathbb{R}$, where $\Phi_{q,r}(a)$ is the mapping from \eqref{eq-2.14}.

If for any $r\in[0,1)$, $x\not=re_{n}$,
then we choose a unitary transformation $A$ such that $A(|x|e_{n})=x$.
For $y\in\mathbb{B}^{n}$, let $W(y)=:u(A(y))$.
Then by \cite[Theorem 5.3.5]{sto2016}, we have
$$
W=P_h[\phi]\circ A=P_h[\phi\circ A].
$$
Since $\|\phi\circ A\|_{L^{p}}=\|\phi\|_{L^{p}}$,
by replacing $u$ with $u\circ A$ and replacing $\phi$ with $\phi\circ A$, respectively,
the similar reasoning as above shows that \eqref{eq-2.18} holds true.

Further, for $r\in[0,1)$ and $q\in(1,\infty)$,
by \eqref{eq-1.2}, \eqref{eq-2.14} and Lemma \ref{lem-2.4}, we get
%\begin{equation}\label{eq-2.19}
$$\min_{a\in \mathbb{R}}\Phi_{q,r}(a)
=\Phi_{q,r}(a^*)
=\Phi_{q,r}\big(a(r)\big)=G_p(r).
$$
This, together with \eqref{eq-2.18},
implies that \eqref{eq-2.17} holds true.
Further, by Lemma \ref{lem-2.4} and \eqref{eq-2.15}, we know that $G_p(0)=0$ and $G_p(r)$ is smooth in $(0,1)$.

Now, we show that the inequality \eqref{eq-2.17} is sharp.
Since $u(0)=0$ and $G_p(0)=0$,
then the equality in  \eqref{eq-2.17} holds for $x=0$.
If $x\in\mathbb{B}^{n}\backslash \{0\}$,
let
\begin{equation}\label{eq-2.19}
\phi_{x}(\eta)
=|P_h(x, \eta)-a(|x|)|^{q/p}
\mathrm{sign}\big( P_h(x,\eta)-a(|x|) \big)
\end{equation}
in $\mathbb{S}^{n-1}$
and define
$$
 u_{x}(y)=P_{h}[\phi_{x}](y)
 $$
 in $\mathbb{B}^{n}$.
It follows from \eqref{eq2.01} and \eqref{eq-2.16} that for any $\rho=|x|\in(0,1)$,
%\begin{equation}\label{eq-2.20}
 $$ F\big(\rho,a(\rho)\big)
  =\int_{\mathbb{S}^{n-1}}
  \big(P_h(x,\eta)-a(\rho) \big)
  |P_h(x,\eta)-a(\rho)|^{q-2} d\sigma(\eta)=0.
$$
Therefore, $u_{x}(0)=P_{h}[\phi_{x}](0)=F\big(\rho,a(\rho)\big)=0$,
and so, for  any $y\in \mathbb{B}^{n}$,
$$
u_{x}(y)
=P_{h}[\phi_{x}](y)
=\int_{\mathbb{S}^{n-1}} \big(P_h(y,\eta)-a\big)\phi_{x}(\eta) d\sigma(\eta).
$$
Let $y=x$.
Then by spherical coordinate transformation \eqref{eq-1.2} and \eqref{eq-2.19}, we get
\begin{eqnarray*}\label{eq-2.022}
u_{x}(x)&=&\int_{\mathbb{S}^{n-1}} |P_h(x,\eta)-a(\rho)|^{q} d\sigma(\eta)
=\int_{\mathbb{S}^{n-1}} |P_h(\rho e_{n},\eta)-a(\rho)|^{q} d\sigma(\eta)\\
&= &\left(\int_{\mathbb{S}^{n-1}} |P_h(\rho e_{n},\eta)-a(\rho)|^qd\sigma(\eta)\right)^{1/q}\|\phi_{x}\|_{L^{p}}
=G_p(|x|)\|\phi_{x}\|_{L^{p}},
\end{eqnarray*}
which means that \eqref{eq-2.17} is an equality
for $u_{x}$ at $x$.
The sharpness of inequality \eqref{eq-2.17} follows.
\end{proof}

\begin{lemma}\label{lem-2.6}
For $p\in(1,\infty]$, suppose $u=P_{h}[\phi]$ and $u(0)=0$, where
$\phi\in L^p(\mathbb{S}^{n-1},\mathbb{R}^{n})$.
Then
\begin{equation}\label{eq-2.20}
 \|D u(0)\|\le 2(n-1) \alpha_{q}^{\frac{1}{q}}\|\phi\|_{L^{p}},
\end{equation}
where $\alpha_{q}=\frac{\Gamma\left[\frac{n}{2}\right] \Gamma\left[\frac{1+q}{2}\right]}{\sqrt \pi\Gamma\left[\frac{n+q}{2}\right]}$.
The inequality is sharp.
 \end{lemma}

\begin{proof}
Let $\phi=(\phi_{1},\ldots,\phi_{n})$ and $u=(u_{1},\ldots,u_{n})$.
For $r\in[0,1)$ and $i\in\{1,\ldots,n\}$, the similar reasoning as in the proof of \cite[Lemma 4.3]{chen2018} shows the gradients
\[\begin{split}
\nabla u_{i}(re_{n})
=\int_{\mathbb{S}^{n-1}} \nabla P_{h}(re_{n},\eta)\phi_{i}(\eta) d\sigma(\eta),
\end{split}\]
where
$$
\nabla P_{h}(re_{n},\eta)
=\frac{2(n-1)(1-r^{2})^{n-2}\big((1-r^{2})(\eta-re_{n})-re_{n}|\eta-re_{n}|^{2}\big)}
{|\eta-re_{n}|^{2n}}.
$$
Then for any $\xi\in\mathbb{S}^{n-1}$, we have
$$
Du(0)\xi
=\big(\langle\nabla u_{1}(0),\xi\rangle,\ldots,\langle\nabla u_{n}(0),\xi\rangle\big)^{T}
=2(n-1)\int_{\mathbb{S}^{n-1}} \langle\eta,\xi\rangle\phi(\eta) d\sigma(\eta),
$$
where $T$ is the transpose. Since
\[\begin{split}
\max_{\xi\in\mathbb{S}^{n-1}}\left|\int_{\mathbb{S}^{n-1}} \langle\eta,\xi\rangle \phi(\eta) d\sigma(\eta)\right|
&\leq  \max_{\xi\in\mathbb{S}^{n-1}} \left(\int_{\mathbb{S}^{n-1}} |\langle\eta,\xi\rangle|^{q} d\sigma(\eta)\right)^{q} \|\phi\|_{L^{p}}\\
&= \left(\int_{\mathbb{S}^{n-1}}
| \eta_{n}|^{q}d\sigma(\eta)\right)^{q}
\|\phi \|_{L^{p}}\\
& =\alpha_{q}^{\frac{1}{q}}\|\phi\|_{L^{p}},
\end{split}\]
where $\eta=(\eta_{1},\ldots,\eta_{n})\in\mathbb{S}^{n-1}$.
Therefore,
$$
\|Du(0)\|=\max_{\xi\in\mathbb{S}^{n-1}}|Du(0)\xi|
\leq2(n-1)\alpha_{q}^{\frac{1}{q}}\|\phi\|_{L^{p}}.
$$

To prove the sharpness of inequality \eqref{eq-2.20},
for any $i\in\{1,\ldots,n\}$, let
$$
\phi(\eta)=|\eta_{i}|^{\frac{q}{p}} \cdot\mathrm{sign} (\eta_{i}).
$$
Then $u=P_{h}[\phi]$ is a mapping from $\mathbb{B}^{n}$ into $ \mathbb{R}$.
By using spherical coordinate transformation, we obtain
$$
u(0)=\int_{\mathbb{S}^{n-1}} \phi(\eta)d\sigma(\eta)
=\int_{\mathbb{S}^{n-1}} |\eta_{n}|^{\frac{q}{p}} \cdot \left( \chi_{\mathbb{S}_{+}^{n-1}} -\chi_{\mathbb{S}_{-}^{n-1}}
\right) d\sigma(\eta)=0
$$
and
\[\begin{split}
| \nabla u(0)|
&=2(n-1)\cdot\max_{\xi\in\mathbb{S}^{n-1}}\left|\int_{\mathbb{S}^{n-1}} \langle\eta,\xi\rangle \phi(\eta) d\sigma(\eta)\right|\\
&\geq 2(n-1) \left|\int_{\mathbb{S}^{n-1}} \langle\eta,e_{i}\rangle \phi(\eta) d\sigma(\eta)\right|
=2(n-1)\alpha_{q}^{\frac{1}{q}}\|\phi\|_{L^{p}}.
\end{split}\]
So the sharpness of \eqref{eq-2.20} follows.
\end{proof}

The following two results are some properties of $G_p$.
 \begin{lemma}\label{lem-2.7}
For $p\in(1,\infty)$,  $G_p:[0,1)\to [0,\infty)$  is a increasing diffeomorphism  with $G_p(0)=0$, where $G_p$ is the mapping from Theorem \ref{thm-1.1}
 \end{lemma}
\begin{proof}
By \eqref{eq-1.2}, it is easy to see that $G_p(0)=0$.
In order to prove
$\lim_{r\to 1^{-}}G_p(r)=\infty$,
we let $u\in \mathcal{H}^p(\mathbb{B}^n,\mathbb{R}^n)\setminus  \mathcal{H}^{\infty}(\mathbb{B}^n,\mathbb{R}^n)$.
Then \eqref{eq-1.1} yields that
$\phi\in L^p(\mathbb{S}^{n-1},\mathbb{R}^n)\setminus L^{\infty}(\mathbb{S}^{n-1},\mathbb{R}^n)$.
These, together with, Lemma \ref{lem-2.5}, imply that $G_p(r)$ is smooth in $(0,1)$ and
\begin{equation}\label{eq-2.21}
\sup_{r\in[0,1)}G_p(r)\ge \frac{\sup_{x\in \mathbb{B}^n}|u(x)|}{\|\phi\|_{L^{p}}}=\infty .
\end{equation}

\begin{claim}\label{claim-2.3}
For $p\in(1,\infty)$, $G_p$ is strictly increasing in [0,1).
\end{claim}
For $p\in(1,\infty)$, $r\in(0,1)$ and $\eta\in \mathbb{S}^{n-1}$,
define
$$
\phi_{*}(\eta)=\frac{|P_h(re_{n},\eta)-a(r)|^{q/p}
\mathrm{sign}(P_h(re_{n},\eta)-a(r))}{\left(\int_{\mathbb{S}^{n-1}}|P_h(re_{n},\eta)-a(r)|^{q} d\sigma(\eta)\right)^{\frac{1}{p}}}.
$$
Then
$\|\phi_{*}\|_{L^{p}}=1$.
For $y\in\mathbb{B}^n$, we let
$u_{*}(y)=P_{h}[\phi_{*}](y)$.
By replacing $u_{x}$ with $u_{*}$ and replacing $\phi_{x}$ with $\phi_{*}$, respectively,
the similar reasoning as in the proof of Lemma \ref{lem-2.5} shows that
$u_{*}(0)=0$ and
$G_p(r)=|u_{*}(re_{n})|$.
Further, \eqref{eq-1.3} and $\|\phi\|_{L^{p}}=1$ imply that $G_p(r)\geq\max_{|x|\leq r}|u_{*}(x)|$.
Therefore,
$$G_p(r)=\max_{|x|\leq r}|u_{*}(x)|=|u_{*}(re_{n})|.$$
Since $u_{*}(0)=0$, by calculations, we get
$$
u_{*}(re_{n})=\left(\int_{\mathbb{S}^{n-1}}|P_h(re_{n},\eta)-a(r)|^{q} d\sigma(\eta)\right)^{\frac{1}{q}}.
$$
Obviously,  $u_{*}$ is not a constant function in any disk $\mathbb{B}_{\rho}^n=\{x\in\mathbb{B}^n:|x|<\rho\}$, where $\rho\in(0,1)$.
Therefore, the maximum principle (cf. \cite[Theorem 4.4.2(a)]{sto2016}) implies that for any $0<r<s<1$,
$$
|u_{*}(0)|=G_p(0)<G_p(r)=\max_{|x|\leq r}|u_{*}(x)|<\max_{|x|\leq s}| u_{*}(x)|=G_p(s).
$$
Hence $G_p$ is a strictly increasing function in [0,1).
This, together with \eqref{eq-2.21}, implies that
$\lim_{r\to 1}G_p(r)=\infty$.
The proof of the lemma is completed.
\end{proof}

\begin{lemma}\label{lem-2.8}
For $p\in(1,\infty )$ and $r\in[0,1)$,
$G_{p}(r)$ is derivable at $r=0$ and
$$
G'_p(0)
 =2(n-1)\alpha_{q}^{\frac{1}{q}},
$$
where $\alpha_{q}$ is the constant from Lemma \ref{lem-2.6}.
 \end{lemma}
\begin{proof}
Let $u=P_{h}[\phi]$, where $\phi \in L^p(\mathbb{S}^{n-1},\mathbb{R})$ and $u(0)=0$.
Then for any $x\in \mathbb{B}^{n}$,
\begin{equation}\label{eq-2.22}
u(x)=\nabla u(0)x+o(x),
\end{equation}
where $o(x)$ is a vector satisfying $\lim_{|x|\rightarrow0^{+}}\frac{|o(x)|}{|x|}=0$.
Then \eqref{eq-2.22}, together with the fact $u\in C^{2}(\mathbb{B}^{n},\mathbb{R})$,
implies that
$$
|\nabla u(0)|=|\langle\nabla u(0),\xi_{0}\rangle|
=\lim_{r\to 0^{+}}\frac{|u(r\xi_{0})|}{r},
$$
where $\xi_{0}=\frac{\nabla u(0)}{|\nabla u(0)|}$.
Further, by Lemma \ref{lem-2.5}, we know that for any $x\in \mathbb{B}^{n}$,
$$
|u(x)|\leq G_p(r)\|\phi\|_{L^{p}},
$$
where $r=|x|$ and $G_p(r)$ is a smooth mapping in $(0,1)$.
Then
\begin{equation}\label{eq-2.23}
|\nabla u(0)|
=\lim_{r\to 0^{+}}\frac{|u(r\xi_{0})|}{r}
=\liminf_{r\to 0^{+}}\frac{|u(r\xi_{0})|}{r}
\leq \liminf_{r\to 0^{+}}\frac{G_p(r)}{r}\|\phi\|_{L^{p}}.
\end{equation}
Let $q=\frac{p}{p-1}$, and for  $\eta=(\eta_{1},\ldots,\eta_{n})\in\mathbb{S}^{n-1}$, define
$$\phi^{*}(\eta)=\alpha_{q}^{-\frac{1}{p}}|\eta_{i}|^{\frac{q}{p}} \cdot\mathrm{sign} (\eta_{i}),
$$
where $i\in\{1,\ldots,n\}$.
Obviously, $\|\phi^{*}\|_{L^{p}}=1$.
By the similar reasoning as in the proof of Lemma \ref{lem-2.6}, we know that the mapping $u^{*}=P_{h}[\phi^{*}]$  satisfies $u^{*}(0)=0$ and
$$
|\nabla u^{*}(0)|=2(n-1) \alpha_{q}^{\frac{1}{q}}.
$$
This, together with \eqref{eq-2.23}, yields
\begin{equation}\label{eq-2.24}
2(n-1) \alpha_{q}^{\frac{1}{q}}
=|\nabla u^{*}(0)|\leq\liminf_{r\to 0^{+}}\frac{G_p(r)}{r},
\end{equation}

On the other hand, since Lemma \ref{lem-2.7} tells us that $G_p(r)$ is a increasing diffeomorphism in $(0,1)$ with $G_p(0)=0$,
then for $x\in\mathbb{B}^{n}\backslash\{0\}$ and $\eta\in\mathbb{S}^{n-1}$,
we define
$$
\phi^{*}_{x}(\eta)
=\big(G_{p}(r)\big)^{-\frac{q}{p}}\cdot|P_h(x,\eta)-a(r)|^{q/p}
\mathrm{sign}\big( P_h(x,\eta)-a(r) \big)
$$
where $r=|x|$.
Obviously, $\|\phi^{*}_{x}\|_{L^{p}}=1$.
By  the similar reasoning as in the proof of Lemma \ref{lem-2.5}, we know that
the mapping
$u^{*}_{x}(y)=P_{h}[\phi^{*}_{x}](y)$ satisfies $u^{*}_{x}(0)=0$ and
$$
u^{*}_{x}(x)=G_p(r).
$$
Then for $x\in\mathbb{B}^{n}\backslash\{0\}$,
by replacing $u$ with $u^{*}_{x}$ in \eqref{eq-2.22},
we obtain
\begin{equation}\label{eq-2.25}
 G_p(r)
=\big|u^{*}_{x}(x)\big|
=\big|\langle\nabla u^{*}_{x}(0),x\rangle+o(x)\big|.
\end{equation}
Further, for $x\in\mathbb{B}^{n}\backslash\{0\}$,
by Lemma \ref{lem-2.6}, we have
$$
 |\nabla u^{*}_{x}(0)|\leq 2(n-1) \alpha_{q}^{\frac{1}{q}}.
$$
This, together with \eqref{eq-2.25}, yields
\begin{eqnarray}\label{eq-2.26}
\limsup_{r\rightarrow0^{+}}\frac{G_p(r)}{r}
=\limsup_{r\rightarrow0^{+}}
\left|\left\langle\nabla u^{*}_{x}(0) ,\frac{x}{|x|}\right\rangle
\right|
\leq 2(n-1) \alpha_{q}^{\frac{1}{q}}.
\end{eqnarray}

By \eqref{eq-2.24} and \eqref{eq-2.26}, we have
$$
G'_p(0) =\lim_{r\rightarrow0^{+}}\frac{G_p(r)}{r}
=2(n-1) \alpha_{q}^{\frac{1}{q}}.
$$
The proof of the lemma is completed.
\end{proof}

\medskip

 As a corollary of our main result, we establish the following counterpart of \cite[Corollary 2.3]{kalaj2018}.
 \begin{corollary}\label{cor-2.1}
Suppose $u=P_{h}[\phi]$ with $\phi\in L^2(\mathbb{S}^{n-1},\mathbb{R}^{n})$.
Then
$$
 \|Du(0)\|\le \sqrt{2(n-1)}\sqrt{\|u\|^2_2-|u(0)|^2}.
$$
 \end{corollary}
 \begin{proof}
 Let $w(x)=u(x)-u(0)$.
Then $w=P_{h}[\phi-u(0)]$ and
$$
\|\phi-u(0)\|_{2}^{2}
=\|\phi\|^2_2+|u(0)|^2-2\int_{\mathbb{S}^{n-1}}
\left\langle\phi(\eta),u(0)\right\rangle d\sigma(\eta)
=\|\phi\|^2_2-|u(0)|^2.
$$
Then by \eqref{eq-1.1} and \eqref{eq-1.4}, we get
$$
\|Du(0)\|=\|Dw(0)\|
\leq   \sqrt{2(n-1)}\|\phi-u(0)\|_2
 = \sqrt{2(n-1)}\sqrt{\|u\|_2^2-|u(0)|^2},
$$
as required.
 \end{proof}
 \section{Special cases}\label{spc}

\subsection{The case $p=\infty$}
For any $x\in \mathbb{B}^{n}$,
let $A$ be an unitary transformation such that $A(re_{n})=x$, where $r=|x|$.
Since $u(0)=0$ and $u=P_{h}[\phi]$, we have
$$
u(x)=u(A(re_{n}))
=\int_{\mathbb{S}^{n-1}} \left(P_h\big(A(re_{n}),\eta\big) - \frac{(1-r^{2})^{n-1}}{(1+r^{2})^{n-1}} \right)\phi(\eta)\;d\sigma(\eta).
$$
For any $\eta\in \mathbb{S}^{n-1}$, let $\xi=A^{-1}\eta$.
Then
\begin{eqnarray}\label{eq-3.1}
|u(x)|
&\leq& \|\phi\|_{L^{\infty}}\int_{\mathbb{S}^{n-1}} \left|P_h(re_{n},\xi) -\frac{(1-r^{2})^{n-1}}{(1+r^{2})^{n-1}} \right| \;d\sigma(\xi)\\\nonumber
&=&\|\phi\|_{L^{\infty}}\int_{\mathbb{S}_{+}^{n-1}} \left(P_h(re_{n},\xi) -\frac{(1-r^{2})^{n-1}}{(1+r^{2})^{n-1}} \right) \;d\sigma(\xi)\\\nonumber
&\;\;&+\|\phi\|_{L^{\infty}}\int_{\mathbb{S}_{-}^{n-1}} \left(\frac{(1-r^{2})^{n-1}}{(1+r^{2})^{n-1}}-P_h(re_{n},\xi) \right) \;d\sigma(\xi)\\\nonumber
&=&U_{h}(re_{n})\cdot\|\phi\|_{L^{\infty}},
\end{eqnarray}
where $U_{h}$ is the mapping in Theorem \ref{thm-1.1}.
By letting
$$
\phi(\eta)
=C\cdot\mathrm{sign}\left(P_h\big(re_{n},A^{-1}\eta\big) - \frac{(1-r^{2})^{n-1}}{(1+r^{2})^{n-1}} \right),
$$
in $\mathbb{B}^{n}$,
we obtain the sharpness of \eqref{eq-3.1},
 where $C$ is a constant.

Further, by Lemma \ref{lem-2.6},
we obtain that the inequality \eqref{eq-1.4} holds for $p=\infty$
and this inequality is also sharp.

Next, we discuss the property of $G_{\infty}$.
It follows from \eqref{eq-3.1} that
\begin{eqnarray}\label{eq-3.2}
G_{\infty}(r)=\int_{\mathbb{S}^{n-1}} |P_h(re_{n},\eta)-a^{*}| d\sigma(\eta)=U_{h}(re_{n}),
\end{eqnarray}
where
$a^*=\frac{(1-r^{2})^{n-1}}{(1+r^{2})^{n-1}}$.
Obviously, $G_{\infty}(0)$=0.
For $x\in \mathbb{B}^{n}$, since $U_{h}(x)$ is a hyperbolic harmonic mapping,
we see that $G_{\infty}(r)=U_{h}(re_{n})$ is differentiable in [0,1).
Moreover, by replacing
$\phi_{*}= \chi_{\mathbb{S}_{+}^{n-1}} -\chi_{\mathbb{S}_{-}^{n-1}} $ and
$u_{*}=  U_{h}$,
respectively,
the similar reasoning as in the proof of Claim \ref{claim-2.3} shows that
$G_{\infty}(r)$ is an increasing diffeomorphism in [0,1).
Since \eqref{eq-1.1}, \eqref{eq-3.1} and \eqref{eq-3.2} yield
$$
1\geq \sup_{r\in[0,1)}U_{h}(re_{n})
=\sup_{r\in[0,1)}G_{\infty}(r)\geq
 \frac{ \|U_{h} \|_{\mathcal{H}^{\infty}}}{ \| \chi_{\mathbb{S}_{+}^{n-1}} -\chi_{\mathbb{S}_{-}^{n-1}}  \|_{L^{\infty}}}=1,
$$
we see that $\lim_{r\rightarrow1^{-}}G_{\infty}(r)=1$.
 Therefore, $G_{\infty}$ maps $[0,1)$ onto $[0,1)$.

In the following, we compute the values of $U_{h}(re_{n})$
(or $G_{\infty}(r)$), where $r\in[0,1)$.
By using spherical coordinate transformation, we obtain that
\begin{eqnarray*}
U_{h}(r e_{n})
&=&(1-r^2)^{n-1} \frac{\Gamma(\frac{n}{2}) }{ \sqrt{\pi} \Gamma(\frac{n-1}{2})} \int_{0}^{\pi}
\frac{\sin^{n-2}\theta}{(1+r^{2}-2r\cos\theta)^{n-1}}(\chi_{\mathbb{S}_{+}^{n-1}} -\chi_{\mathbb{S}_{-}^{n-1}}) \;d\theta.
\end{eqnarray*}
 Elementary calculations lead to
\begin{eqnarray*}
&\;\;& \int_{0}^{\pi}
\frac{\sin^{n-2}\theta}{(1+r^{2}-2r\cos\theta)^{n-1}}(\chi_{\mathbb{S}_{+}^{n-1}} -\chi_{\mathbb{S}_{-}^{n-1}}) \;d\theta\\
&=& \int_{0}^{\frac{\pi}{2}}
\left(\frac{\sin^{n-2}\theta}{(1+r^{2}-2r\cos\theta)^{n-1}}-\frac{\sin^{n-2}\theta}{(1+r^{2}+2r\cos\theta)^{n-1}}\right) \;d\theta\\
&=&\frac{1}{(1+r^2)^{n-1}} \sum_{k=0}^{\infty} \int_{0}^{\frac{\pi}{2}}\sin^{n-2}\theta \cos^{k}\theta  d\theta \cdot
\left(
\begin{array}{c}
-(n-1) \\
k \\
\end{array}
\right)
 \cdot \big( (-1)^{k}-1\big) \cdot\left( \frac{2r}{1+r^{2}}\right)^{k} ,
\end{eqnarray*}
where
$$
\left(
\begin{array}{c}
-(n-1) \\
k \\
\end{array}
\right)
 =\frac{(-1)^{k}\Gamma(k+n-1)}{k!\Gamma(n-1)}.
 $$
Since
 $$
 \int_{0}^{\frac{\pi}{2}}\sin^{n-2}\theta \cos^{k}\theta  d\theta= \frac{\Gamma(\frac{1+k}{2}) \Gamma(\frac{n-1}{2})}{ 2 \Gamma(\frac{n+k}{2})}
 $$
(cf. \cite[Page 19]{rain}), then
$$
U_{h}(r e_{n})
= \frac{2r(1-r^2)^{n-1}\Gamma(\frac{n}{2})}{\sqrt{\pi}(1+r^{2})^{n}\Gamma(n-1)}
\sum_{k=0}^{\infty}\frac{\Gamma(k+1) \Gamma(2k+n)}{\Gamma(k+\frac{n+1}{2})  (2k+1)!} \left(\frac{2r}{1+r^{2}}\right)^{2k}.
$$
By Legendre's duplication formula (cf. \cite[Page 24]{rain}), we get
$$
\frac{\Gamma(2k+n)}{ (2k+1)!}
=\frac{2^{2k+n-1} \Gamma(k+\frac{n}{2}) \Gamma(k+\frac{n}{2}+\frac{1}{2})}{\sqrt{\pi}(2k+1)!}
=\frac{2^{n-1} \Gamma(k+\frac{n}{2}) \Gamma(k+\frac{n}{2}+\frac{1}{2})}{\sqrt{\pi} \cdot(\frac{3}{2})_{k} \cdot k!}.
$$
Therefore,
$$
U_{h}(r e_{n})
=\frac{2^{n}r(1-r^2)^{n-1}\Gamma(\frac{n}{2})}{ \pi(1+r^{2})^{n}\Gamma(n-1)}
\sum_{k=0}^{\infty}\frac{ \Gamma(k+1)\Gamma(k+\frac{n}{2}) \Gamma(k+\frac{n}{2}+\frac{1}{2})}{\Gamma(k+\frac{n+1}{2}) \cdot(\frac{3}{2})_{j} \cdot k! }\left(\frac{4r^{2}}{(1+r^{2})^{2}}\right)^{k},
$$
which, together with the fact $\Gamma(k+\alpha)=(\alpha)_{k}\cdot \Gamma(\alpha)$ for any $\alpha>0$,
implies that
\begin{eqnarray*}
U_{h}(r e_{n})
&=&\frac{2^{n}r(1-r^2)^{n-1}\Gamma^{2}(\frac{n}{2})}{ \pi(1+r^{2})^{n}\Gamma(n-1)}
\sum_{k=0}^{\infty} \frac{(1)_{k} (\frac{n}{2})_{k}}{(\frac{3}{2})_{k} k!}\left(\frac{4r^{2}}{(1+r^{2})^{2}}\right)^{k}\\
&=&\frac{2^{n}r(1-r^2)^{n-1}\big(\Gamma(\frac{n}{2})\big)^{2}}{ \pi(1+r^{2})^{n}\Gamma(n-1)}
\;_{2}F_{1} \left(1,\frac{n}{2};\frac{3}{2}; \frac{4r^{2}}{(1+r^{2})^{2}}\right).
 \end{eqnarray*}

The following table shows first few functions $U_{h}(r e_{n})$, where $r\in[0,1)$.

 \begin{center}
\tabcolsep=12.5pt
\renewcommand\arraystretch{1.5}
\begin{tabular}{|c|c|c|c|c| }\hline
$n$ & $2$ & $3$ & $4$ & $5$ \\ \hline
$U_{h}(r e_{n})$ & $ \frac{4}{\pi}\arctan r$ & $ \frac{2r}{1+r^{2}}$ & $  \frac{4r(1-r^{2})}{\pi(1+r^{2})^{2}}+\frac{4}{\pi} \arctan r $ & $ \frac{3r+2r^3+3r^5}{(1+r^{2})^3}$  \\ \hline
\end{tabular}
 \medskip

 Table 1. Value of $U_{h}(r e_{n})$.
\end{center}

\medskip
\medskip

\subsection{The case $p=2$}
In this case we deal with the extremal problem
$$
G_2(r)
=\left(\inf_{a\in\mathbb{R}}
\int_{\mathbb{S}^{n-1}} |P_h(re_{n},\eta)-a|^2 d\sigma(\eta)\right)^{1/2}.
$$
By \cite[Equality (2.6) and Theorem G]{chen2018}, we obtain that
\begin{eqnarray*}
&\;&
\int_{\mathbb{S}^{n-1}}|P_h(re_{n},\eta)-a|^2 d\sigma(\eta)\\
&=&\int_{\mathbb{S}^{n-1}} P^{2}_{h}(rN,\eta) d\sigma(\eta)+a^2\int_{\mathbb{S}^{n-1}}  d\sigma(\eta)
-2 a \int_{\mathbb{S}^{n-1}} P_h(re_{n},\eta) d\sigma(\eta)\\
&=&(1-r^2)^{2n-2}F\left(2n-2,\frac{3n-2}{2};\frac{n}{2};r^2\right)+a^2-2a.
\end{eqnarray*}
So $a^*=1$ and
$$
\left(\inf_a \int_{\mathbb{S}^{n-1}} |P_h(re_{n},\eta)-a|^2 d\sigma(\eta)\right)^{1/2}
=\sqrt{(1-r^2)^{2n-2}\cdot F\left(2n-2,\frac{3n-2}{2};\frac{n}{2};r^2\right)-1}.$$
This, together with \eqref{eq-2.17}, implies that
$$
|u(x)|\leq G_2(r)\cdot\|\phi\|_{L^{p}}=\|\phi\|_{L^{p}}\sqrt{(1-r^2)^{2n-2}\cdot F\left(2n-2,\frac{3n-2}{2};\frac{n}{2};r^2\right)-1}.
$$
\subsection{The case $p=1$}
In this case we have
$$
G_1(r)=\inf_{a\in \mathbb{R}} \sup_{\eta\in\mathbb{S}^{n-1}} |P_h(re_{n},\eta)-a|
$$
for any $r\in[0,1)$.
Since
$$
\max_{\eta\in\mathbb{S}^{n-1}}  P_h(re_{n},\eta)=\frac{(1+r)^{n-1}}{(1-r)^{n-1}}
\quad\text{and}\quad
 \min_{\eta\in\mathbb{S}^{n-1}} P_h(re_{n},\eta)=\frac{(1-r)^{n-1}}{(1+r)^{n-1}},
$$
we easily conclude that
$$
a^* =\frac{1}{2}\left(\frac{(1+r)^{n-1}}{(1-r)^{n-1}}+\frac{(1-r)^{n-1}}{(1+r)^{n-1}}\right)
\; \text{and}\;
G_1(r)=\frac{1}{2}\left(\frac{(1+r)^{n-1}}{(1-r)^{n-1}}-\frac{(1-r)^{n-1}}{(1+r)^{n-1}}\right).$$
Observe that $G_1(r) = |P_h(re_{n},e_n)-a^*|=|P_h(re_{n},-e_n)-a^*|$, and this fact is important to construct the minimizing sequence.
Obviously, $G_{1}(r)$ is an increasing diffeomorphism from $[0,1)$ onto $[0,\infty)$.
Then for any $x\in \mathbb{B}^{n}$ with $|x|=r$, we have
\begin{eqnarray}\label{eq-3.3}
|u(x)|
\leq G_1(r)\cdot\|\phi\|_{L^{1}}
=\frac{1}{2}\left(\frac{(1+r)^{n-1}}{(1-r)^{n-1}}-\frac{(1-r)^{n-1}}{(1+r)^{n-1}}\right)\|\phi\|_{L^{1}}.
\end{eqnarray}
This, together with the fact $u(x)=D u(0)x+o(x)$,
implies
\begin{eqnarray}\label{eq-3.4}
\|Du(0)\|
\leq \limsup_{x\rightarrow0}\frac{|u(x)|}{|x|}
\leq \limsup_{r \rightarrow0^{+}}\frac{G_1(r)}{r}\|\phi\|_{L^{1}}
=2(n-1)\|\phi\|_{L^{1}}.
\end{eqnarray}

Now, we show the sharpness of \eqref{eq-3.3} and \eqref{eq-3.4}.
For $i\in\mathbb{Z}^{+}$, $\eta\in\mathbb{S}^{n-1}$ and $x\in\mathbb{B}^{n}$, we let
$$
\phi_{i}(\eta)=\frac{ \chi_{\Omega_{i}}(\eta)}{2\| \chi_{\Omega_{i}} \|_{L^{1}}}-\frac{\chi_{\Omega^{'}_{i}}(\eta) }{2\| \chi_{\Omega'_{i}}\|_{L^{1}}}
\quad\text{and}\quad
u_{i}(x)=P_{h}[\phi_{i}](x),
$$
where $\Omega_{i}=\{\eta\in\mathbb{S}^{n-1}:|\eta-e_{n}|\leq\frac{1}{i}\}$ and
$\Omega^{'}_{i}=\{\eta\in\mathbb{S}^{n-1}:|\eta+e_{n}|\leq\frac{1}{i}\}$.
Obviously,
$$ \|\phi_{i}\|_{L^{1}}=1,\quad \int_{\mathbb{S}^{n-1}}\phi_{i}(\eta)d\sigma(\eta)=0
\quad
 \text{and}
 \quad
 u_{i}(0)=0.
 $$
By elementary calculations, we get
\begin{eqnarray}\label{eq-3.5}
\lim_{i\rightarrow\infty} u_{i}(x)
 &=& \lim_{i\rightarrow\infty}\int_{\mathbb{S}^{n-1}}(P_{h}(x,\eta)-a^{*})
 \left(\frac{ \chi_{\Omega_{i}}(\eta)}{2\| \chi_{\Omega_{i}} \|_{L^{1}}}-\frac{\chi_{\Omega^{'}_{i}}(\eta)}{2\| \chi_{\Omega'_{i}} \|_{L^{1}}} \right)
  d\sigma(\eta)\\\nonumber
&=&\lim_{i\rightarrow\infty} \int_{\mathbb{S}^{n-1}}|P_{h}(x,\eta)-a^{*}|
\left(\frac{ \chi_{\Omega_{i}}(\eta)}{2\| \chi_{\Omega_{i}} \|_{L^{1}}}+\frac{\chi_{\Omega^{'}_{i}}(\eta)}{2\| \chi_{\Omega'_{i}} \|_{L^{1}}} \right)
d\sigma(\eta).
\end{eqnarray}

\begin{claim}\label{claim-3.1} For any $r\in[0,1)$,
$$
\lim_{i\rightarrow\infty}\int_{\mathbb{S}^{n-1}}|P_{h}(re_{n},\eta)-a^{*}|\cdot
 \frac{ \chi_{\Omega_{i}}(\eta)}{ \| \chi_{\Omega_{i}} \|_{L^{1}}}d\sigma(\eta)
 =G_{1}(r).
$$
\end{claim}
Observe that
$$
 \lim_{i\rightarrow\infty}|P_{h}(re_{n},\eta)-a^{*}| \cdot \chi_{\Omega_{i}}(\eta)=G_{1}(r).
$$
Then for any $\varepsilon>0$, there exists a positive integer $m_{1}=m_{1}(\varepsilon)$ such that for any $i\geq m_{1} $,
$$
\big\|P_{h}(re_{n},\eta)-a^{*}| \cdot \chi_{\Omega_{i}}(\eta)-G_{1}(r)\big|<\varepsilon.
$$
Since
$\int_{\mathbb{S}^{n-1}}
\frac{ \chi_{\Omega_{i}}(\eta)}{\|\chi_{\Omega_{i}}\|_{L^{1}}}
d\sigma(\eta)=1$, then for any $i\geq m_{1} $ and $r\in[0,1)$,
\begin{eqnarray*}
&\;\;&\left|\int_{\mathbb{S}^{n-1}} |P_{h}(re_{n},\eta)-a^{*}|\cdot
 \frac{ \chi_{\Omega_{i}}(\eta)}{ \| \chi_{\Omega_{i}} \|_{L^{1}}}d\sigma(\eta)
 -G_{1}(r)\right|\\
& =&\int_{\mathbb{S}^{n-1}} \big\|P_{h}(re_{n},\eta)-a^{*}|\cdot \chi_{\Omega_{i}}(\eta) -G_{1}(r)\big|
 \cdot\frac{ \chi_{\Omega_{i}}(\eta)}{ \| \chi_{\Omega_{i}} \|_{L^{1}}}d\sigma(\eta)
 \leq\varepsilon,
\end{eqnarray*}
which means that the claim is true.

The similar reasoning as in the proof of Claim \ref{claim-3.1} shows that
$$
\lim_{i\rightarrow\infty}\int_{\mathbb{S}^{n-1}}|P_{h}(re_{n},\eta)-a^{*}|\cdot
 \frac{ \chi_{\Omega'_{i}}(\eta)}{ \| \chi_{\Omega'_{i}} \|_{L^{1}}}d\sigma(\eta)
 =G_{1}(r).
$$
This, together with \eqref{eq-3.5}, Claim \ref{claim-3.1} and the fact $\|\phi_{i}\|_{L^{1}}=1$, shows that
$$
\lim_{i\rightarrow\infty}|u_{i}(x)|
=G_{1}(r)\lim_{i\rightarrow\infty}
\|\phi_{i}\|_{L^{1}}.
$$
The sharpness of \eqref{eq-3.3} follows.

Further, since the similar reasoning as in the proof of Claim \ref{claim-3.1} implies
$$
\lim_{i\rightarrow\infty}\int_{\mathbb{S}^{n-1}} |\eta_{n}|\cdot
 \frac{ \chi_{\Omega_{i}}(\eta)}{ \| \chi_{\Omega_{i}} \|_{L^{1}}}d\sigma(\eta)
=
\lim_{i\rightarrow\infty}\int_{\mathbb{S}^{n-1}} |\eta_{n}|\cdot
 \frac{ \chi_{\Omega'_{i}}(\eta)}{ \| \chi_{\Omega'_{i}} \|_{L^{1}}}d\sigma(\eta)
  =1,
$$
then by similar similar arguments as in the proof of Lemma \ref{lem-2.6},
we get
\begin{eqnarray*}
\lim_{i\rightarrow\infty} \|\nabla u_{i}(0)\|
&\geq&\lim_{i\rightarrow\infty}
2(n-1) \left|\int_{\mathbb{S}^{n-1}} \langle\eta,e_{n}\rangle \phi_{i}(\eta) d\sigma(\eta)\right|\\
&\geq&2(n-1) \lim_{i\rightarrow\infty}\int_{\mathbb{S}^{n-1}} |\eta_{n}|\cdot \left(\frac{ \chi_{\Omega_{i}}(\eta)}{2\| \chi_{\Omega_{i}} \|_{L^{1}}}+\frac{\chi_{\Omega^{'}_{i}}(\eta)}{2\| \chi_{\Omega'_{i}} \|_{L^{1}}} \right) d\sigma(\eta) \\
&=&2(n-1)\lim_{i\rightarrow\infty} \|\phi_{i}\|_{L^{1}},
\end{eqnarray*}
which shows the sharpness of \eqref{eq-3.4}.

\vspace*{5mm}
\noindent{\bf Funding.}
 The first author was partly supported by NSFS of China (No. 11571216, 11671127 and 11801166),
 NSF of Hunan Province (No. 2018JJ3327), China Scholarship Council and the construct program of the key discipline in Hunan Province.

\end{document}